\documentclass[12pt]{article}
\usepackage{latexsym,color,amsmath,amsthm,amssymb,amscd,amsfonts}
\usepackage{scalefnt}
\usepackage[font=small,labelfont=bf]{caption}
\usepackage{float}
\usepackage{graphicx}
\usepackage{amsopn}
\usepackage{relsize}
\usepackage{mathrsfs}
\usepackage{upgreek}
\usepackage{tikz}
\usepackage{subcaption}
\usepackage{caption} 
\usetikzlibrary{calc}

\newtheoremstyle{nonum}{}{}{\itshape}{}{\bfseries}{.}{ }{\thmnote{#3}}

\newtheorem{thm}{Theorem}[section]
\newtheorem*{thm*}{Theorem}

\newtheorem{cor}[thm]{Corollary}
\newtheorem{lem}[thm]{Lemma}
\newtheorem{prop}[thm]{Proposition}
\newtheorem{rem}[thm]{Remark}

\newtheorem*{definition*}{Definition}

\newtheorem*{rems*}{Remarks}

\theoremstyle{nonum}

\newtheorem*{definition-recall}{Definition \ref{def:klens}}

\newcommand{\R}{\mathbb R}
\newcommand{\RR}{\mathbb R}

\newcommand{\N}{\mathbb N}

\def\K{{\cal K}}

\def\S{{\cal S}}

\newcommand{\iprod}[2]{\langle #1,#2 \rangle} 

\def\K{{\cal K}}

\def\vol{{\rm Vol}}

\def\eps{{\varepsilon}}

\def\conv{{\rm conv}}
\def\cconv{{\rm conv}_c}
\def\outrad{{\rm Outrad}}

\begin{document}
\title{Symmetrizations of Ball-Bodies}
\date{}
\author{S. Artstein-Avidan\thanks{Support for this work was provided by the ISF grant number 1626/25.} \and D.I. Florentin\footnotemark[1]}
\maketitle
\begin{abstract}
We study symmetrization procedures within the class $\mathcal S_n$ of \emph{ball-bodies}, i.e.\ intersections of unit Euclidean balls (equivalently, summands of the Euclidean unit ball, or $c$-convex sets via the $c$-duality $A\mapsto A^c$). We first examine linear parameter systems obtained by replacing the usual convex hull by the $c$-hull $A^{cc}$, deriving consequences for volume along these $c$-paths.  In particular, we obtain convexity statements in special cases and in dimension $2$, and we show by example that such convexity fails in general for $n\ge 3$. We then focus on Steiner symmetrization. We prove that Steiner symmetrization increases the \emph{dual volume} and that in the planar case Steiner symmetrals of ball-bodies remain ball-bodies. In contrast, we provide an explicit example in $\RR^3$ showing that the Steiner symmetral of a ball-body need not belong to $\mathcal S_n$, and show that there are such counter-examples with arbitrarily large curvatures.
\end{abstract}

\section{Introduction and preliminaries}\label{sec:first}

The class $\S_n$ of ball-bodies in $\RR^n$, discussed in \cite{AFpreprint}, can be defined in various equivalent forms. It consists of bodies which have generalized sectional curvatures at least $1$ at every boundary point, and this is equivalent to these sets being summands of the Euclidean ball $B(0,1)$. In other words, for every $K\in \S_n$ there corresponds some $L\in \S_n$ with $K+L = B(0,1)$. Equivalently, the class can be defined as the image class of the order reversing quasi involution on subsets of $\RR^n$, taking a set $A\subseteq \RR^n$ to its so called ``$c$-dual set'' given by 
\[ A^c = \bigcap_{x\in A} B(x,1). \]
This class has recently received considerable attention, 
and is connected with an array of problems and conjectures  such as Borisenko's conjecture \cite{BorisenkoDrach2014,ChernovDrachTatarko2019,Drach2016,drach2023reverse} and the Blaschke-Lebesgue problem  \cite{blaschke1915breite,lebesgue1914isoperimetres,MartiniMontejanoOliveros2019}, see \cite{AFpreprint} for more details. 

It is not hard to show (see \cite{AFpreprint}) that for any $A\subset \R^n $ we have $A\subseteq A^{cc}$ which, (together with order reversion) in particular implies that on the image, namely on sets which are intersections of translates of the Euclidean unit ball,  the mapping $A\mapsto A^c$ is an order reversing involution (see \cite{artstein2023zoo}). It is also an isometry with respect to the Hausdorff metric, and up to rigid motions no other isometry, except the identity, exists \cite{ACF1, ACF2}. 
We call the smallest element in $\S_n$ which includes a set $A$ the $c$-hull of $A$ and note that it is simply given by 
$\cconv (A) : = A^{cc}$. 
We will make use of the following near-linearity of the $c$-duality, which is immediate when one realizes for $K\in\S_n$ its $c$-dual set is   $K^c = -L$
where $K+L = B(0,1)$. 
Let $K,T\in \S_n\setminus \{\emptyset, \RR^n\}$,  and $\lambda \in (0,1)$. Then $(1-\lambda)K + \lambda T\in \S_n$, and
\begin{equation}\label{eq:lin}
((1-\lambda)K + \lambda T)^c = (1-\lambda)K^c + \lambda T^c.
\end{equation}
Moreover,  even if $K,T\subset \R^n$ 
are not assumed to be in $\S_n$, we still have 
 \begin{equation}\label{eq:minkowski-sum-c-is-sub-linear}
    (1-\lambda)K^c + \lambda T^c
    \subseteq
    ((1-\lambda)K + \lambda T)^c, 
\end{equation}   
which follows from the definition of the dual, since $K\subseteq B(x,1)$ and $T\subseteq B(y,1)$ implies $(1-\lambda)K + \lambda T \subseteq B((1-\lambda)x+\lambda y, 1)$. 

The tool of Steiner symmetrization, and more generally of linear parameter systems, is very useful in convexity, see e.g. \cite[Section 1.1.7]{AGMBook}. In this note  we gather information regarding the application of this process in the context of $\S_n$. It should be mentioned that many variants of symmetrizations exist, in particular Minkowski symmetrizations and some in-between hybrids, see \cite{BianchiGardnerGronchi}. Some of the results in this note apply to these inbreeds as well, but we discuss mainly the extremal ones - Minkowski and mainly Steiner. 

After recalling some tools as preliminaries, we discuss the concept of linear parameter systems when one replaces the usual convex hull with the $c$-hull. We show some convexity properties of the volume in special cases, as well as in dimension $n=2$, and explain that in dimension $n\ge3$ convexity of the volume of the $c$-hull fails in general. 

In Section \ref{sec:Steiner} we turn to Steiner symmetrization within $\S_n$: we show that it increases the dual volume, prove that in the planar case Steiner symmetrization preserves the class $\S_2$, and then give an explicit lens example in $\RR^3$ showing that Steiner symmetrization need not preserve $\S_3$ (equivalently, it may produce boundary points with sectional curvature $<1$). We proceed to showing that in fact even in the subclass of bodies in $\RR^n$ $n\ge 3$ which are intersections of $r$-balls for arbitrarily small $r>0$, we can find an element and a Steiner symmetrization which maps it to a body which is not a ball-body.

\section{Symmetrizations of Ball-Bodies}

Dealing with a certain class, it is natural to search for simple operations on the class which, applied sequentially, say, take a set to be more symmetric, which could mean invariant under certain linear transformations, or, usually, closer to a ball.  

A very simple example which works for ball bodies is the operation of Minkowski symmetrizations (see \cite[Section 1.5.5]{AGMBook}.
For $K\subset \R^n$ and $u\in S^{n-1}$, the reflection of $K$ with respect to the hyperplane $u^\perp$ is denoted by $R_u K$. The Minkowski symmetral $M_u K$ of $K$ with respect to $u$, is defined to be the Minkowski average of $K$ and its reflection, i.e. $M_u K=\frac12K + \frac12 R_u K$. Clearly, the Brunn-Minkowski inequality implies that $\vol(M_u K) \ge \vol(K)$. Minkowski symmetrizations commute with $c$-duality, and in particular they preserve the class of ball-bodies, that is

\begin{lem}\label{lem:closed-under-Mink}
Let $u\in S^{n-1}, K\in \S_n$. Then $M_u K\in\S_n$, and $M_u(K^c) = (M_u K)^c$. Thus the notation $M_u K^c$ is unambiguous, and moreover $\vol(M_u K^c) \ge \vol(K^c)$.
\end{lem}

In particular, one may use Lemma \ref{lem:closed-under-Mink} to prove that among all bodies $K\in \S_n$ with fixed mean-width $w(K) = \int_{S^{n-1}} h_K(u)d\sigma(u)$ (here $\sigma$ is the normalized Haar measure on the sphere $S^{n-1}$), the one with maximal volume is the Euclidean unit ball, and the one with maximal dual volume is {\em also} the Euclidean unit ball. The first of these two facts is true also for any $K\in \K^n$ (all convex bodies) due to Urysohn's inequality, \[ \left(\vol(K)/\vol(B_2^n)\right)^{1/n} \le w(K)/w(B_2^n).  \]
However the inequality for dual volume is specialized to this class both because the $c$-duality is relevant only for this class, and because $w(K^c)  = 1- w(K)$, again by the linearity, which is very much not the case for usual polarity. 

\begin{cor}\label{cor:santalofixedwidth}
Let $A\subset \RR^n$ with out-radius at most $1$, and let $B(0,r)$ be the Euclidean unit ball with the same mean-width as $\cconv (A)$. Then  
\begin{equation}\label{eq:santalo-fixed-vol}
  \vol(A^c) \le   \vol(B(0,1-r)) = (1-r)^n \kappa_n.
\end{equation} 	 
\end{cor}

Minkowski symmetrization is a useful operation, but somewhat ``too linear'' for many applications. It is the {\em Steiner} symmetrization (see \cite[Section 1.1.7]{AGMBook} which is mostly used in Asymptotic Convex Geometry. We recall the definition here, and show some simple properties which hold in general, and get back to it in Section \ref{sec:Steiner}.  Fix a unit vector $u\in S^{n-1}$. 
The Steiner symmetrization $S_u(K)$ of a convex body $K$ with respect to the hyperplane $u^\perp$ is defined as follows. For any $x$ in the projection $P_{u^\perp}(K)$ of $K$ to $u^\perp$, we denote $K\cap\left(x+\R u\right)=[x+a(x)u,x+b(x)u]$ for ``the segment above $x$'', of length $|b(x)-a(x)|$. The Steiner symmetral of $K$ with respect to $u$ is defined to be
\begin{equation*}
S_u(K)=\left\{(x,y)\in u^\perp \times \RR\,:\, x\in P_{u^\perp}(K),\,|y|\le\frac{|b(x)-a(x)|}{2} \right\}.
\end{equation*}

\begin{thm}\label{thm:steiner-increases-dual-volume}
	Let $K\subset \RR^n$ be convex and let $u\in S^{n-1}$. Then 
	\[ S_u (K^c) \subseteq \cconv(S_u (K^c))\subseteq  (S_u K )^c.\]
	In particular, Steiner symmetrizations increase $c$-dual volume; $\vol(K^c) \le 	\vol((S_u K)^c)$.
\end{thm}

\begin{proof}
We will use Minkowski symmetrization, first noting that for any convex $K$, one has $S_u (K) \subseteq M_u (K)$. Indeed,
in $S_u(K)$ the segment above a given $x\in P_{u^\perp}(K)$ is the Minkowski average of the segments above $x$ of $K$ and its reflection, which is in $M_u(K)$. Secondly, using   \eqref{eq:minkowski-sum-c-is-sub-linear} we see that 
\[
M_u (K^c) = 
\frac{K^c + R_u(K^c)}2 = 
\frac{K^c + (R_u K)^c}2 \subseteq 
\left(\frac{K + R_u K}2 \right)^c =
(M_u K)^c
\]
(in the second equality we use that $(U A )^c = U (A^c)$ for any isometry $U$). 
Joining these two facts, and the fact that $c$-duality reverses inclusion we get 
\[
S_u (K^c) \subseteq
M_u (K^c) \subseteq
(M_u K)^c \subseteq
(S_u K)^c. 
\]
Since the right hand side belongs to $\S_n$, inclusion remains also after taking the $c$-hull of the left hand side, which completes the proof. 
\end{proof}

We can  deduce a Santal\'{o}-type inequality, although it is formally weaker than Corollary \ref{cor:santalofixedwidth}. 

\begin{cor}\label{cor:santalofixedvol}
Let $A\subset \RR^n$ with out-radius at most $1$, and let $B(0,r)$ be the Euclidean unit ball with the same volume as $\conv (A)$. Then  
\begin{equation}\label{eq:santalo-fixed-vol2}
  \vol(A^c) \le   \vol(B(0,1-r)) = (1-r)^n \kappa_n.
\end{equation} 	 
\end{cor}

\begin{proof}
One may find a sequence of Steiner symmetrizations of $\conv(A)$ which converges to a ball of the same volume (see e.g. \cite[Theorem 1.1.16]{AGMBook}). Using Theorem \ref{thm:steiner-increases-dual-volume}, the volume product is increasing along the sequence, which completes the proof.
\end{proof}

Before moving on to linear parameter systems, which will allow us in particular to work in more detail with with Steiner symmetrization, we mention yet another symmetrization that was used in the literature, also for the class $\S_n$. 
In \cite{Bezdek2018} 
Bezdek proves a fact similar to Corollary \ref{cor:santalofixedvol} using a symmetrization called ``two-point symmetrization''. To describe it, denote for an affine hyperplane $H$ define the operation of reflection with respect to $H$ by $R_H$.  The two-point symmetral of $K$ with respect to $H$ is
\[ \tau_H (K) = (K\cap R_H(K)) \cup ((K\cup R_H(K))\cap H^+). \]
It is easy to check that $K$ and $\tau_H(K)$ have the same volume (but convexity, as well as $c$-convexity, need not be preserved). 
\begin{thm}[Bezdek]
	If $K\subset \RR^n, n>1$  then 
	\[ \tau_H(K^c) \subseteq
    {\rm conv}_c \left(\tau_H(K^c)\right) \subseteq
    (\tau_H (K))^c. \] 
	In particular, among all compact sets of a given volume, the ball has the largest (in volume) $c$-dual. 
\end{thm}
Bezdek used this theorem to prove a special case of the Knesser-Poulsen conjecture (See \cite{bezdek2008kneser}).

\section{Linear parameter systems}\label{sec:lps}

Steiner symmetrizations can be thought of as a time $1$ map for a parametric family of bodies. These are best expressed via ``shadow systems'' (thinking of the shadow an object casts as a function of time when the sun moves) or of ``linear parameter system''. We describe these in the classical setting as well as the ball-bodies setting. 

For a set $A\subset \RR^n$, a vector $v\in \RR^n$, and a ``velocity function'' $\alpha:A\to \RR$, we let 
\[ A_t =  \{x+ t \alpha(x)v: x\in A\},\qquad K_t =  {\rm conv} (A_t),\]
and 
\[ L_t =  \cconv (A_t).\]

In classical convexity theory, the set $K_t$ is called a linear parameter system, and these were investigated in depth by Rogers and Shephard \cite{shephard1964shadow, rogers1958extremal}, where for example it was shown that $\vol(K_t)$ is a convex function of $t$ (as are all other quermassintegrals of $K_t$, see e.g.~\cite[Theorem 10.4.1]{schneider2013convex}). The following proposition can be seen as analogous to the fact that $1/\vol(K_t^\circ)$ is a convex  function of $t$ (with the usual polar), which was
proved by Campi and Gronci in \cite{CampiGronchi2006}.

\begin{prop}
Let $n\in {\mathbb N}$,  $A\subset \RR^n$,  $v\in \RR^n$ and  $\alpha:A\to \RR$. For $ t_\lambda  = (1-\lambda)t_0 + \lambda t_1$ it holds that 
	\[ L_{t_\lambda} \subseteq (1-\lambda)L_{t_0} + \lambda L_{t_1}\qquad {\rm and} \qquad (1-\lambda)  L_{t_0} ^c + \lambda  L_{t_1} ^c \subseteq L_{t_\lambda}^c.\]
	In particular, $\vol(L_t^c)^{1/n}$ is  concave in $t$, as are the quermassintegrals $V_k(L_t^c)^{1/k}$.
\end{prop}

\begin{proof}
Denote $A_\lambda = A_{t_\lambda}$ and $L_\lambda = L_{t_\lambda}$. Since by definition $A_{ \lambda} \subseteq (1-\lambda)A_0 + \lambda A_1$, it holds that 
    $A_{ \lambda} \subseteq (1-\lambda)L_0 + \lambda L_1$ and the right hand side belongs to $\S_n$ by \eqref{eq:lin}. Therefore we also have that  $L_{ \lambda} \subseteq (1-\lambda)L_0 + \lambda L_1$, proving the first inclusion.

    Applying $c$-duality to both sides, and using \eqref{eq:lin} again, we get $(1-\lambda)L_0^c + \lambda L_1^c \subseteq  L_{t_\lambda}^c$ as claimed. The corresponding concavity follows from the Brunn-Minkowski inequality. 
\end{proof}

It is natural to ask  whether in analogy to the classical theorem of Rogers and Shephard, the function $\vol(L_t)$ is convex. This question is intimately tied with the question of whether Steiner symmetrization preserves the class $\S_n$. As we shall demonstrate shortly (in Theorem \ref{thm:RSconvexityofvolR2} and Section \ref{sec:stein-not-in-class}), the answer is that this is true in dimension $n\le 2$ and false in higher dimensions. However, when considering a system of just {\em two} points, the answer is yes in any dimension, as the following proposition states.

\begin{prop}\label{prop-vol-of-carambula-is-convex}
Let $n\ge 2$. For any $y_0 \in \RR^n$, the function
$F:\RR^n \to [0,\infty]$ given by 
$F(x) =  \vol \left(\cconv(x,y_0)\right)$  is convex.
\end{prop}

\begin{proof}
Letting $d = \|x-y_0\|/2$, the volume of the $1$-lens $\cconv(x,y_0)\subseteq \RR^n $ is given (up to a multiplicative constant $2\kappa_{n-1}$) by the function
\begin{equation}\label{eq:vol-of-car} F_n(d) =\int_0^d
\left(\sqrt{1-t^2} - \sqrt{1-d^2}    \right)^{n-1} dt\end{equation}
for $d\le 1$, and $F_n(d)=\infty$ otherwise. To see that $F_n$ is convex, it suffices to check that $F_n((1-\lambda)d_0 + \lambda d_1) \le (1-\lambda) F_n(d_0 ) + \lambda F_n(d_1)$ for $d_i\in(0,1)$. To this end differentiate
\begin{eqnarray*} F_n'(d) &=& \int_0^d (n-1) \left(\sqrt{1-t^2} - \sqrt{1-d^2}    \right)^{n-2} \cdot( d (1-d^2)^{-1/2} )dt + 0
	\\& =&(n-1)d (1-d^2)^{-1/2}  F_{n-1}(d) > 0
\end{eqnarray*}
(here $n\ge2$ and $F_1(d)=d$). Since $F_n'$ is the product of two (or three) positive increasing functions, it is increasing,
which means $F_n$ is convex. Since the function
$F(x)=F_n(\|x-y_0\|/2)$ is the composition of the convex function $x\mapsto \|x-y_0\|/2$ with the increasing convex function $F_n$, it is convex, which completes the proof.
\end{proof}

Proposition \ref{prop-vol-of-carambula-is-convex} implies that for a $c$-shadow system of two points, the volume is indeed a convex function. More precisely,
\begin{cor}\label{cor-2-pt-chadow-system-has-convex-vol}
Let $n\ge 2$, $v\in S^{n-1}$, $x_0,y_0\in \R^n$, $\alpha,\beta\in\R$, and for each $t\in\R$ let 
\[ L_t =  \cconv \{x_0+ t \alpha v, y_0+ t \beta v\}.\]
Then the function $f(t):=\vol_n(L_t)$ is convex.
\end{cor}
\begin{proof}
As in Proposition \ref{prop-vol-of-carambula-is-convex}, we denote the half-diameter of the $1$-lens $L_t$ by 
\[d(t):=\|x(t)-y(t)\|/2,\]
where $x(t)=x_0 + t\alpha v$ and $y(t)=y_0 + t\beta v$. The function $F_n:[0,\infty) \to [0,\infty]$ (from Proposition \ref{prop-vol-of-carambula-is-convex}) is convex and increasing, and the function $d:\R\to[0,\infty) $ is convex, thus the composition $f=F_n\circ d$ is convex.
\end{proof}

\subsection{The planar case}

In this section we show that in dimension $n=2$, the volume of the $c$-hull of a linear parameter system is indeed convex. 

\begin{thm}\label{thm:linparsysR2}
Let $A\subset \RR^2$ be a bounded set,  let $v\in S^1$ and let  $\alpha : A\to \RR$. For each $t\in \RR$ let $ L_t =  \cconv \{x+ t \alpha(x)v: x\in A\}$. Then the function $F(t)  = \vol (L_t )$ is convex.  
\end{thm}

The proof of Theorem \ref{thm:linparsysR2} follows from the case of a finite set $A$ (Theorem \ref{thm:RSconvexityofvolR2}) together with a limiting argument. For the case of a finite $A$ we will use induction on the number of points, combined with the following technical lemma regarding ``locally convex enlargement'' of a convex function.

\begin{lem}\label{lem:patching-convexity}
Let $g:\R\to (-\infty,\infty]$ be a convex function, and let $f:T\to (-\infty,\infty]$, where $T\subset \R$ is an open set, i.e. $T$ is a countable union of pairwise disjoint intervals $\{(a_n,b_n)\}_{n\in \N}$. Assume that for every such interval $(a_n,b_n)$, the restriction $f\big|_{(a_n,b_n)}$ is a convex function that agrees with $g$ at the endpoints of the interval i.e. 
$\lim_{t\to a_n^+}f(t)=g(a_n)$, and $\lim_{t\to b_n^-}f(t)=g(b_n)$. Then the function
\[
h(t)=
\begin{cases}	\max\{f(t),g(t)\}&  t\in T  \\
	        g(t)& t\notin T
\end{cases}
\]
is convex in $\R$.
\end{lem}

\begin{proof}
Deonte $T_n=\cup_{k=1}^n (a_k,b_k)$ and let 
\[
h_n(t)=
\begin{cases}	\max\{f(t),g(t)\}&  t\in T_n  \\
	        g(t)& t\notin T_n
\end{cases}
\]
Clearly $h$ is the point-wise limit of $(h_n)_{n=1}^\infty$, thus it suffices to show that $h_n$ is convex. Since convexity is a local property, and it holds both 
on the open set $T_n$ and on open subsets of its complement, we need only check it at $\partial T_n$. Since $\partial T_n  = \{a_k,b_k\}_{k\le n}$ we may without loss of generality check convexity around the point $t = a_1$. The right derivative of $h_n$ at $a_1$ exists and is equal to the right derivative of $\max(f,g)$ (which is a convex function on $(a_1, b_1)$) at $a_1$. For some $\varepsilon>0$ we have $(a_1-\varepsilon, a_1)\subset \mathbb{R}\setminus T_n$ and $(a_1, a_1+\varepsilon)\subset T_n$, thus $h_n=g$ on $(a_1-\varepsilon, a_1)$ and we get
\begin{eqnarray*}
h_n'(a_1^-) &=& g'(a_1^-)\le g'(a_1^+) =
\lim_{\delta\to 0^+} \frac {g(a_1 + \delta) - g(a_1)}{\delta} \\
&\le &\lim_{\delta\to 0^+} \frac {h_n(a_1 + \delta) - g(a_1)}{\delta} =
\lim_{\delta\to 0^+} \frac {h_n(a_1 + \delta) - h_n(a_1)}{\delta} = h_n'(a_1^+),
\end{eqnarray*}
completing the proof.
\end{proof}

The boundary structure of a ball body can be studied and analyzed in a manner similar to that of convex bodies in $\RR^n$, by replacing convex hull with $c$-hull. For an elaborate discussion of this theme, explaining the similarities and differences, see \cite{AFpreprint}. For our purposes in this note, we only need a simple observation regarding {\em $c$-extremal points} (the $c$-hull variant of extreme points), namely that $c$-extremal points of the $c$-hull $\cconv(A)$, must belong to the original set $A$.

In Proposition \ref{prop:cara-hard-neeD} we give the definition of $c$-extremal points, together with the above mentioned statement. For its proof the reader is referred to \cite{AFpreprint} (although we shall only use it in the planar case, where the proof is very easy).  

\begin{prop}\label{prop:cara-hard-neeD}
Let $n\ge 2 $ and $K\in \S_n$. A point $x\in K$ is called $c$-extremal for $K$ if $x\in {\conv}_c (y,z)$ for $y,z\in K$ implies $y = x$ or $z = x$. We denote the set of $c$-extremal points of $K$ by ${\rm ext}_c(K)$. Let $A\subset \R^n$ be closed with $\outrad(A)<1$.
Then
\[
{\rm ext}_c(\cconv(A)) \subseteq A.
\]
\end{prop}

With these preliminaries in hand, we can provide the  proof for the following ball-polytope  case of Theorem \ref{thm:linparsysR2}.

\begin{thm}\label{thm:RSconvexityofvolR2}
Let $v\in S^{1}$, $A=\{ x_i\}_{i=1}^m \subset \RR^2$ and $\{ \alpha_i\}_{i=1}^m \subset \RR$. For each $t\in \RR$ let 
\[ L_t =  \cconv \{x_i+ t \alpha_i v : i=1, \ldots m\}.\]
 Then the function $h(t) = \vol(L_t)$ is convex.    
\end{thm}
\begin{proof}
We prove by induction on the number of points $m$, where the base case $m=2$ was handled in Corollary \ref{cor-2-pt-chadow-system-has-convex-vol}.
For every strict subset $I\subsetneq \{1,\dots,m\}$ we define $f_I:\R\to [0,\infty]$ by $f_I(t)=\vol(\cconv\{x_i+t\alpha_iv\,:\, i\in I\})$, and  define $g:\R\to [0,\infty]$ by $g=\sup_I \{f_I\}$, with the supremum running over all strict subsets of $\{1,\dots,m\}$. 
By the induction hypothesis, $f_I$ are convex on $\R$, and thus also $g$ is convex. 

By Proposition \ref{prop:cara-hard-neeD} we have ${\rm ext}_c(L_t)\subseteq \{x_i+ t \alpha_i v\}_{i=1}^m$ for every $t\in \RR$, and we define the set $T\subset \R$ to be the set of all $t\in\R$ for which there is equality in this inclusion, i.e. ${\rm ext}_c(L_t) = \{x_i+ t \alpha_i v\}_{i=1}^m$. 
The set $T$ is open and may be empty. If it is empty then by the induction hypothesis we are done, as $h = g$ everywhere. Assume $T$ is nonempty and define the function $f:T\to  [0,\infty]$ by $f(t)=\vol (L_t)$, i.e. $f = h|_T > g|_T$.
Consider an interval $(a,b)\subset T$.
Since on this interval ${\rm ext}_c(L_t) = \{x_i + t\alpha_i v\}_{i=1}^m$, the set $L_t$ is a disk-polygon, namely consists of a polygon $\conv \{x_i+t\alpha_iv\}_{i=1}^m$
and $m$ halves of $1$-lenses between neighboring vertices. Each of these sets has volume which is convex in $t$, and therefore the union has volume which is convex in $t$.  In other words, on the interval $(a,b)$, the function $f$ is convex. We thus satisfy the conditions of Lemma \ref{lem:patching-convexity} and conclude that $h(t)=\vol (L_t)$ is a convex function on $\RR$. 
 \end{proof}

Clearly $c$-polytopes are Hausdorff-dense in ball-bodies (since for a convex $K$ one can find a polytope $P\subseteq K$ which is arbitrarily close to it in the Hausdorff metric, and its $c$-hull will remain a subset of $K$ if $K\in \S_n$, giving a $c$-polytope which is at least as close). We can use this to  conclude convexity of the volume for general $c$-linear parameter systems in dimension $n = 2$ (the fact that  $\vol(L_t)$ may fail to be convex when working in dimension $n\ge 3$ will follow from Section \ref{sec:stein-not-in-class}). The only missing tool is the following proposition regarding continuity of the $c$-hull with respect to  the Hausdorff metric. This is quite intuitive and was proved in \cite[Corollary 2.11]{AFpreprint}.
 
\begin{prop}\label{prop:c-hull-is-continuous}
	Let $n\in \N$, and consider the class of non-empty subsets of $\RR^n$ with out-radius at most $1$. On this class, the mapping $K\mapsto K^{cc} = \cconv(K)$ is continuous in the Hausdorff metric.
    \end{prop}

\begin{proof}[Proof of Theorem \ref{thm:linparsysR2}]
    Recall $A_t = \{ x+ t\alpha(x)v:x\in A\}$. First note that the function $R_t = \outrad(A_t)$ is convex. Indeed, the inclusion
    $A_{t_\lambda} \subseteq (1-\lambda)A_{t_0} + \lambda A_{t_1}$ clearly holds by definition for $t_{\lambda} = (1-\lambda)t_0 +\lambda t_1$, and since the out-radius of a set is convex with respect to Minkowski addition, we see that $R_t$ is convex. Therefore, there is an interval $[t_{\min}, t_{\max}]$ where $(A_t)^{cc} \neq \RR^n$, and $L_{t_{\min}}, L_{t_{\max}}$ are unit Euclidean balls (there is only one case where $t_{\min} = -\infty$ and $t_{\max} = +\infty$, namely when $\alpha$ is a constant function, as in all other cases there are two points moving in different velocities, meaning for large enough $|t|$, the out-radius of $A_t$ is more than $1$).
We may thus restrict to times $t\in [t_{\min}, t_{\max}]$, as out of this interval $F(t) = +\infty$. 

Let $t_0, t_1\in [t_{\min}, t_{\max}]$ and $\lambda \in (0,1)$, and let 
$t_{\lambda} = (1-\lambda)t_0 +\lambda t_1$. Let $D_0, D_1, D_\lambda$ be countable subsets of $A$, such that $(D_{i})_t \subseteq A_{t_i}$ is dense for $i = 0,1,\lambda$, and consider $D_0\cup D_1\cup D_\lambda \equiv D \subseteq A$. Let $(A_m)_{m=1}^\infty$ be an increasing sequence, such that $A_m\subseteq D$ consists of $m$ points, and $\cup_{m=1}^\infty A_m = D$. 
By construction, $(A_m)_t \to A_t$ for $t = t_0, t_1, t_\lambda$ as $m\to \infty$, where this limit is in the Hausdorff sense. 

By Proposition \ref{prop:c-hull-is-continuous}, and using that $\outrad((A_m)_t)\le \outrad(A_t)\le 1$, we see that $ \cconv((A_m)_t) \to \cconv (A_t)$ for $t = t_0, t_1, t_\lambda$. We use Theorem \ref{thm:RSconvexityofvolR2} which implies
\[ \vol(\cconv (A_m)_{t_\lambda}) \le (1-\lambda)\vol(\cconv (A_m)_{t_0}) + \lambda \vol(\cconv (A_m)_{t_1}).
\]
Taking the limit $m\to\infty$, and using continuity (with respect to the Hausdorff distance) of volume on the class $\S_2$, we get the desired inequality.
\end{proof}

\section{Steiner symmetrizations and the class $\S_n$}\label{sec:Steiner}

It is well known that, in the notations of Section \ref{sec:lps}, the Steiner symmetral $S_u(K)$ can be realized by a linear parameter system in direction $v=u$, as follows. Letting $A=K$ and assigning velocity $\alpha((x,y))=-(a(x)+b(x))/2$ to all points $(x,y)\in K$ with $y\in [a(x), b(x)]$, we have $K_0 = K$, $K_1=S_u(K)$, and $K_2 = R_u K$.

Since 
for any $t\in[0,2]$ we have $A_t=K_t$, the length of the segment $K_t\cap\left(x+\R u\right)$ is constant for any $x\in P_{u^\perp} (K)$, Fubini's theorem implies that 
the volume of $K_t$ is constant for $t\in [0,2]$.  Clearly $K_t\subseteq L_t$, which ties intimately the question of whether Steiner symmetrization preserves the class $\S_n$ with the question of convexity of $\vol(L_t)$. Indeed, 
if $S_u(K)\not\in \S_n$ then $K_1 \subsetneq L_1$, which means  
\[
\vol(L_1) = \vol(\cconv(S_u(K)))> \vol(S_u(K)) = \vol(L_0) = \vol(L_2).
\]
We will see in Section \ref{sec:steinerintheplane} and Section \ref{sec:stein-not-in-class} that Steiner symmetrization preserves the class $\S_n$ only when $n\le 2$. In particular, this implies that $\vol(L_t)$ cannot  be convex (or even quasi-convex) in general for $n\ge 3$.

\subsection{Steiner Symmetrization of ball-bodies in the plane}\label{sec:steinerintheplane} 

In this section we prove the following theorem.

\begin{thm}\label{thm:InR2Steiner-preserves-class}
	Let $K\in \S_2$ and let $u\in S^{1}$.
	The Steiner symmetral $S_u(K)$ of $K$ in direction $u$   also belongs to $\S_2$. 	
\end{thm}

We shall give two proofs for the theorem. The first relies on the results of the previous section regarding linear parameter systems, whereas the second is analytic and will help us construct the counterexample in dimension $3$. 

\begin{proof}[First proof of Theorem \ref{thm:InR2Steiner-preserves-class}]
Consider the linear parameter system described above, with $A=K$, $v = u$,  $\alpha(x+tu) = -(a(x)+b(x))/2$ for all points $x+tu\in K$ with $x\in P_{u^\perp}K$ and $t\in [a(x), b(x)]$. With this construction we have $K_0 = K$, $K_1=S_u(K)$, and $K_2 = R_u(K)$. Note that by Theorem \ref{thm:linparsysR2}, in $\RR^2$ the function $\vol(L_t)$ is convex, and on the other hand $\vol(L_t) \ge \vol(K_t) = \vol(K)$. Therefore the function $\vol(L_t)$ must be constant on the interval $[0,2]$, and in particular the bodies $L_1$ and $K_1$ coincide, thus $S_u(K) = L_1 \in \S_2$.
\end{proof}

\begin{proof}[Second proof of Theorem \ref{thm:InR2Steiner-preserves-class}]
Besides $\emptyset, \RR^2$, and singletons, sets in $\S_2$ are characterized as closed convex sets for which the generalized curvature at all boundary points is at least $1$. We may without loss of generality assume that $u = e_2$. We first prove the theorem in the case where  $K$ is smooth, with boundary given by the graphs of two twice continuously differentiable concave functions $f$ and $-g$ with some support set $[a,b]$, that is $K=\{(x,y) \,:\, x\in [a,b],\, y\in [f(x),-g(x)] \}$. Since there are no segments on the boundary of a set in $\S_2$, we see that $f(a)=-g(a)$ and $f(b) = -g(b)$. 
 
 For $x\in (a,b)$ 
 the curvatures at the points $(x,f(x))$ and $(x,-g(x))$ are given by $\kappa_f(x) = \frac{|f''(x)|}{(1+(f'(x))^2)^{3/2}}$ and $\kappa_g(x) = \frac{|g''(x)|}{(1+(g'(x))^2)^{3/2}}$ (the general dimensional case is in \eqref{eq:sec-curv} below).  Letting $h(x)=(f(x)+g(x))/2$, the Steiner symmetral of $K$ is given by $S_u(K)=\{(x,y) \,:\, x\in [a,b],\, y\in [h(x),-h(x)] \}$, and the curvature at points $(x,\pm h(x))$ satisfies
	\begin{eqnarray}\label{eq:Steiner-is-in-S_2}
		\kappa_h(x) &=& \frac{|h''(x)|}{(1+(h'(x))^2)^{3/2}}
		=\frac{\frac12|f''(x)+g''(x)|}{\left(1+\left(\frac{f'(x)+g'(x)}{2}\right)^2\right)^{3/2}}\ge \\
		&\ge& \frac{\frac12|f''(x)+g''(x)|}
		{
			\frac12\left(1+f'(x)^2\right)^{3/2} +
			\frac12\left(1+g'(x)^2\right)^{3/2}
		}=
		\frac{|f''(x)|+|g''(x)|}
		{ |f''(x)|/ \kappa_f(x) + |g''(x)|/ \kappa_g(x)} \ge 1.\nonumber
	\end{eqnarray}
	The first inequality holds since the function $t\mapsto (1+t^2)^{3/2}$ is (strictly) convex, and the second inequality holds since $\kappa_f,\kappa_g\ge1$. 
    
We need to also consider the points $x=a$ and $x=b$. Start with the latter. By smoothness, the normal to $K$ at $(b,f(b))$ is in direction $e_1$ and   by assumption, the ball $B((b-1, f(b)), 1)$ contains $K$. Therefore $S_u(K) \subset S_u (B((b-1, f(b)), 1)) = B((b-1, 0), 1)$. This is a $1$-ball supporting $S_u(K)$ at $(b,h(b)) = (b,0)$. The same argument works for $x = a$.
This completes the proof in the case of sufficiently smooth $K$. 

To complete the proof in the general case we use that 
Steiner symmetrization is continuous on bodies with non-empty interior (see e.g. \cite[Proposition A.5.1.]{AGMBook}), and that smooth bodies are dense in $\S_n$. This latter property was proven in \cite{AFpreprint}, and we include it in the appendix as Theorem \ref{thm:dense-are-the-smooth} for completeness. With these in hand, given a body $K\in \S_2$, we take $K_m\to K$ with $K_m\in \S_2$ which are sufficiently smooth, use the first part of the proof to get that $S_u(K_m)\in \S_2$ for all $m$, and the continuity of $S_u$ to get $S_u(K_m) \to S_u(K)$. We use that $\S_2$ is closed in the Hausdorff metric, which can be easily verified, say by using that there is some $T_m$ with $S_u(K_m)+T_m = B(0,1)$, and $T_m$ will have some limit $T$ since $S_u(K_m)$ has a limit, and by definition $S_u (K)+ T = B(0,1)$ we have $S_u (K) \in \S_2$, as claimed.
\end{proof}

\begin{rem}
Theorem \ref{thm:InR2Steiner-preserves-class} has many applications, which are quite direct. We mention one which is of interest, namely the analogue of Macbeath's theorem from \cite{Macbeath1951} regarding best approximation from within.  Of all bodies $K\in \S_2$ of area $V$, the Euclidean disk of area $V$ is hardest to approximate in the following sense. Fixing the amount of vertices $N$, consider a  disk-polygon $P_N$ with $N$  vertices contained in $K$,  for which $\vol(P_N)$ is the largest. This area is minimal when $K$ is a Euclidean ball. Formally: 
\begin{thm}
For $N\in \mathbb N$ and $K\in \S_2$ let  
\[
V_N(K) = \max\{ \vol(P_N): P_N~{\rm has~at~most~}N{\rm ~vertices}\}. 
\]
      Let $K\in \S_2$ 
and let
       $r = (\vol(K)/\pi)^{1/2}$, that is  $\vol(rB_2^2) = \vol(K)$. Then 
        \[ V_N(rB_2^2) \le V_N(K),\qquad {\rm for~all~}N\in \mathbb N.\]
\end{thm}
\noindent     For the proof one simply follows Macbeath's proof from \cite{Macbeath1951}, and use that the body obtained by averaging the upper envelope $g$ and the lower envelope $-f$ must also be in $\S_2$, hence once it contains all the vertices, it also contained their $c$-hull. 
    This observation along with more sophisticated ones will appear in full in future work. 
\end{rem}

\subsection{A counterexample in dimension $3$}\label{sec:stein-not-in-class}

The previous section makes it natural to believe the Steiner symmetrization will preserve the class $\S_n$ is any dimension, since it respects inclusion by balls. However, as we shall see in this section, already in dimension $3$ the symmetral of a set in $\S_3$ might have some sectional curvature smaller than $1$.

In fact, our counterexample will be an $(n-1)$-lens, namely the intersection of two Euclidean unit balls. The fact that once a counter-example exists, such a simple set can serve as a counter-example, follows from the following proposition, which essentially states that if, given some set $K\in \S_n$ and some direction $u\in S^{n-1}$, the family of supporting lenses $L_x$ for $K$ satisfy that all these lenses satisfy $S_u(L_x)\in \S_n$, then $S_u(K) \in \S_n$. In particular, once we know a counterexample exists, there must be one given by a lens. While we could have simply written up the lens counter-example without this proposition, we  include it since in certain cases it can help determine, for a given $K$, that $S_u(K) \in \S_n$.

\begin{prop}\label{prop:lenses-suffice}
Let $K\in \S_n$ and $u\in S^{n-1}$. Assume that for every $x\in P_{u^{\perp}}(K)$, letting $t,s\in \RR$ satisfy $[x+su, x+tu] = K\cap (x+\RR u)$, a supporting lens $L = (z+B_2^n) \cap (w+ B_2^n)$ with $K\subseteq L$ and $x+tu\in \partial L$, $x+su\in \partial L$ exists with $S_u(L)\in \S_n$. Then $S_u(K)\in \S_n$.    
\end{prop}

\begin{rem}
    Note that at least one supporting lens $L = (z+B_2^n) \cap (w+ B_2^n)$ with $K\subseteq L$ and $x+tu\in \partial L$, $x+su\in \partial L$ exists. We phrased the theorem's condition as  the existence of a lens with symmetral in $\S_n$ since there may be more than one supporting lens at $x+tu$ and $x+su$, and it suffices to require only {\em one} of these lenses to have its symmetral be in $\S_n$. 
\end{rem}

\begin{proof}[Proof of Proposition \ref{prop:lenses-suffice}]
    Assume the condition in the statement of the proposition holds, and let $x\in P_uK$. Then $S_u(K)$ has the two boundary points $x \pm ru$ with $r=\frac{(t-s)}{2}$. It suffices to show that at these boundary points the set $S_u(K)$ has a supporting unit Euclidean ball. Consider the lens $L_x = (z(x)+B_2^n )\cap (w(x)+ B_2^n)$ with $K\subseteq L_x$ and $x+tu\in \partial L_x$, $x+su\in \partial L_x$ which is given by the assumption in the proposition. Then $K\cap (x + \RR u) = [x+su, x+tu] = L_x\cap (x+\RR u)$. Therefore 
    \[ S_u (K) \cap  (x + \RR u) =  [x-ru, x+ru] = S_u (L_x) \cap  (x + \RR u).\]
    Also, $S_u(K) \subseteq S_u(L_x)$ by monotonicity of Steiner symmetrization. The assumption $S_u(L_x)\in \S_n$ implies one may find supporting unit Euclideans ball for $S_u(L_x)$ at $x+ru$ and at $x-ru$, and these will also be supporting unit Euclidean balls for $S_u(K)$ at these points, as required. 
\end{proof}

\begin{thm}\label{thm:counter1}
    The lens in $\RR^3$ given by 
    \[ L = B(c_0, 1) \cap B(-c_0,1), \quad c_0 = (-0.2794,0.2451,0.36) \]
satisfies that $S_{e_3}(L)\not\in \S_3$. More specifically,
the sectional curvature in direction $e_1$ at the two points $(0.4154,0.7262,\pm z)\in \partial S_{e_3}(L)$ is smaller than $1$.
\end{thm}

\begin{proof}
Denote $c_0=(x_0,y_0,z_0)$. Then
$B(c_0,1)=\{ (x,y,z):
-f_d(x,y) \le z \le f_u(x,y)
\}$ and likewise $B(-c_0,1)=\{ (x,y,z):
-g_d(x,y) \le z \le g_u(x,y)
\}$, where the functions $f_d,f_u:B\left((x_0,y_0),1 \right)\to\R$ and $g_d,g_u:B\left(-(x_0,y_0),1 \right)\to\R$ are given by
\begin{eqnarray*}
    f_u(x,y) &=& z_0 + \sqrt{1 - (x-x_0)^2 - (y-y_0)^2}\\    
    -f_d(x,y) &=& z_0 - \sqrt{1 - (x-x_0)^2 - (y-y_0)^2}\\
    g_u(x,y) &=& -z_0 + \sqrt{1 - (x+x_0)^2 - (y+y_0)^2}\\
    -g_d(x,y) &=& -z_0 - \sqrt{1 - (x+x_0)^2 - (y+y_0)^2}.
\end{eqnarray*} 
 This means $L = \{ (x,y,z): \max(-f_d(x,y), -g_d(x,y))
\le z \le
\min(f_u(x,y), g_u(x,y))
\}$. (See Figure \ref{fig:functionsonlens}.)
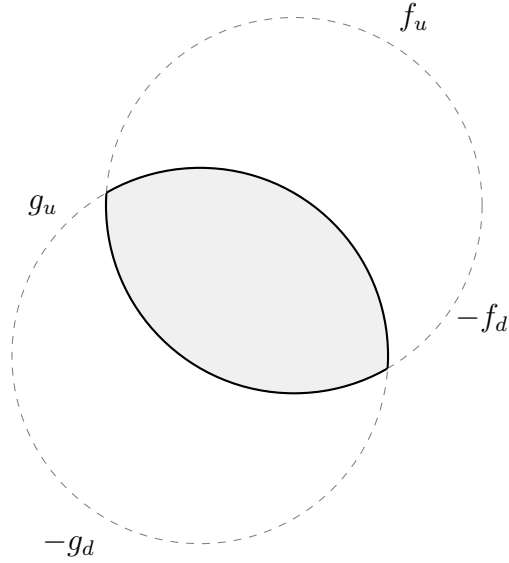
\begin{figure}[H]
\centering
\begin{tikzpicture}[scale=2.5]
  \centering
  \coordinate (C1) at (0.25, 0.4);  
  \coordinate (C2) at (-0.25, -0.4); 
  \def\r{1}
  \begin{scope}
    \clip (C1) circle (\r);
    \fill[gray!12] (C2) circle (\r);
  \end{scope}
  \draw[dashed, gray] (C1) circle (\r);
  \draw[dashed, gray] (C2) circle (\r);
  \begin{scope}
    \clip (C1) circle (\r);
    \draw[line width=0.8pt] (C2) circle (\r);
  \end{scope}
  \begin{scope}
    \clip (C2) circle (\r);
    \draw[line width=0.8pt] (C1) circle (\r);
  \end{scope}
  \node[right] at ($(C1)+(0,\r)+ (0.5,0)$) {$f_u$};
  \node[right] at ($(C1)+(0.8,-\r+0.4)   $) {$-f_d$};
  \node[left]  at ($(C2)+(0,\r)+ (-0.7,-0.2)$) {$g_u$};
  \node[left]  at ($(C2)+(0,-\r)+ (-0.5,0)$) {$-g_d$};
\end{tikzpicture}
\caption{Illustrating the functions}
  \label{fig:functionsonlens}
\end{figure}

Before specifying the values for $x_0, y_0, z_0$, let us mention that in general, points $(x,y)$ in $P_{e_3}(L)$ can be of two different forms. It could be that $-g_d\le -f_d \le f_u \le g_u$ (or similarly $-f_d\le -g_d \le g_u \le f_u$), in which case the fiber above $(x,y)$ is simply the intersection of $(x,y,0) +\RR e_3$ with one of the balls (as is the case for some points in Figure \ref{fig:lens2}). In this case, no curvature is compromised since we are dealing with a region where the set being symmetrized is locally a ball. 

The second and more interesting case is when $-g_d\le -f_d\le g_u \le f_u$ (or similarly $-f_d\le -g_d\le f_u \le g_u$). In such a case, the fiber above $(x,y)$ has one end on $\partial B(c_0, 1)$ and the other end on $\partial B(-c_0, 1)$, as is the case for {\em all} points in Figure \ref{fig:lens1}.

Our choice of the parameters $(x_0, y_0, z_0)$ corresponds to the latter case, meaning that $-g_d\le -f_d\le g_u \le f_u$ holds at the point $(x,y) = (0.4154,0.7262)$, and in fact this holds for all points in $P_{e_3}(L)$ (though the proof does not rely on it).
Indeed, 
\[
-g_d(x,y) \approx -0.56 \, < \, 
-f_d(x,y) \approx -0.175\, < \, 
 g_u(x,y) \approx -0.165\, < \, 
 f_u(x,y) \approx 0.89          
\]  
This also shows that our point $(x,y)$ lies in the projection of the lens. In fact, $(x,y,0)$ is very close to $\partial P_{e_3}(L)$; note the tiny difference between $g_u$ and $-f_d$, meaning that the fiber $((x,y, 0)+ \RR e_3) \cap L = [(x,y,-f_d(x,y)), (x,y,g_u(x,y))]$ is very short. The corresponding fiber in the Steiner symmetral $S_{e_3}(L)$ is $[(x,y,-h(x,y)), (x,y,h(x,y))]$, where $h =\frac{f_d + g_u}{2}$. 

The sectional curvature (see \cite{schneider2013convex}) at the point $(w,H(w))$ on the graph of a function $H:\RR^{2}\to\RR$, in direction $e_1$, is given by 
\begin{equation}\label{eq:sec-curv}
\kappa_H(w,e_1)  = \frac{ \left|(\nabla^2 H(w))_{1,1}\right|}{\sqrt{1+\|\nabla H(w)\|_2^2}\left(1 +\iprod{e_1}{\nabla H(w)}^2\right)}=
\frac{ \left|(\nabla^2 H(w))_{1,1}\right|}{\psi(\nabla H(w))},
\end{equation}
where the function $\psi:\RR^2\to\RR^+$ is defined by $\psi(s,t)= \sqrt{1+s^2+t^2}(1+s^2)$. (This is well known, see e.g. \cite[Sec.4-3]{doCarmo} or \cite[Vol. 3, p. 38]{Spivak}.) Since $g_u$ and $f_d$ correspond to (translated) spheres, for them this expression must equal to $1$, and we only need to use \eqref{eq:sec-curv} for the case $H=h$. We can easily compute
 \begin{eqnarray*}
  \nabla   f_d(x,y) &=& -\frac{((x-x_0), (y-y_0) )}{
  \sqrt{1 - (x-x_0)^2 - (y-y_0)^2}} \approx   -(1.300, 0.900)\\ 
   \nabla  g_u(x,y) &=& -\frac{((x+x_0), (y+y_0) )}{
  \sqrt{1 - (x+x_0)^2 - (y+y_0)^2}}\approx -(0.697, 4.977)
\end{eqnarray*} 
where we use $\approx$ to indicate rounding to the displayed decimal. For $h = (f_d+g_u)/2$ we have $\nabla h  = (\nabla f_d+\nabla g_u)/2$ and $\nabla^2 h  = (\nabla^2 f_d+\nabla^2 g_u)/2$, so we may use \eqref{eq:sec-curv} to get
\begin{eqnarray*}\label{eq:Steiner-wrong1}
	\kappa_h(w,e_1) 
	&=&
    \frac{ \left| (\nabla^2 h(w))_{1,1}\right|}{\psi(\nabla h(w))} =
	\frac
    {  \frac12 \left|\left(\nabla^2 f_d(w)\right)_{1,1}+  (\nabla^2 g_u(w))_{1,1}\right| }
    {\psi\left(\frac{\nabla f_d(w) + \nabla g_u(w)}{2} \right)}.
\end{eqnarray*}
Finally, we arrive at the main observation, which is that $\psi$ is {\em not convex} on $\RR^2$  (this explains why a proof such as that of Theorem \ref{thm:InR2Steiner-preserves-class} fails in $\RR^3$; compare this to \eqref{eq:Steiner-is-in-S_2}, where convexity of the denominator was crucial).
In particular, we chose our parameters so that at the point $w = (x,y)\in \RR^2$ we have
\begin{equation}\label{eq:psi-not-cvx-to-verify}
\psi \left(\frac{\nabla f_d (w) + \nabla g_u(w)}2\right)
\approx 6.51
>
6.32 \approx
\frac{\psi(\nabla f_d(w))+\psi(\nabla g_u(w))}2 .
\end{equation}
Combining this with \eqref{eq:sec-curv} for the concave functions $H=g_u$ and $H=f_d$ (together with $\kappa_{g_u}(w,e_1) =\kappa_{f_d}(w,e_1)=1$), we see that  
\begin{eqnarray*} 
1 & = &\frac{ \frac12 \left|\left(\nabla^2 f_d(w)\right)_{1,1} +  (\nabla^2 g_u(w))_{1,1}\right| }
{ \frac12 \left| \left(\nabla^2 f_d(w)\right)_{1,1}+ (\nabla^2 g_u(w))_{1,1} \right|} =
\frac
{\frac12 \left|\left(\nabla^2 f_d(w)\right)_{1,1} +  (\nabla^2 g_u(w))_{1,1}\right|}
{\frac{1}{2} \left(\psi(\nabla f_d(w))+ \psi(\nabla g_u(w)))\right)} \\ 
&=&
\kappa_h(w, e_1)\cdot \frac
{\psi \left(\frac{\nabla f_d (w) + \nabla g_u(w)}2\right)}
{\left(\frac{\psi(\nabla f_d(w))+ \psi(\nabla g_u(w)))}{2}\right)}
\approx \frac{6.51}{6.32}\cdot \kappa_h(w, e_1)
> \kappa_h(w, e_1),
\end{eqnarray*}
as claimed.

\end{proof}

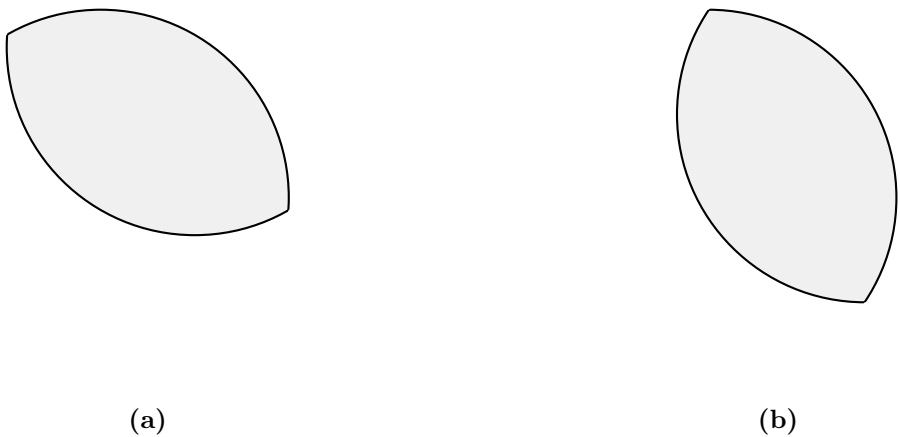
\begin{figure}[H]
  \centering
  \begin{subfigure}[t]{0.45\textwidth}
    \centering
    \begin{tikzpicture}[scale=2.5]
      \def\r{1}
      \coordinate (C1) at (0.25, 0.4);
      \coordinate (C2) at (-0.25, -0.4);
      \begin{scope}
        \clip (C1) circle (\r);
        \fill[gray!12] (C2) circle (\r);
      \end{scope}
      \begin{scope}
        \clip (C1) circle (\r);
        \draw[line width=0.8pt] (C2) circle (\r);
      \end{scope}
      \begin{scope}
        \clip (C2) circle (\r);
        \draw[line width=0.8pt] (C1) circle (\r);
      \end{scope}
    \end{tikzpicture}
    \caption{}
    \label{fig:lens1}
  \end{subfigure}
  \hfill
  \begin{subfigure}[t]{0.45\textwidth}
    \centering
    \begin{tikzpicture}[scale=2.5]
      \pgfmathsetmacro{\ang}{-30}
      \def\r{1}
      \coordinate (C1r) at ({0.25*cos(\ang) - 0.4*sin(\ang)}, {0.25*sin(\ang) + 0.4*cos(\ang)});
      \coordinate (C2r) at ({-0.25*cos(\ang) - (-0.4)*sin(\ang)}, {-0.25*sin(\ang) + (-0.4)*cos(\ang)});

      \begin{scope}
        \clip (C1r) circle (\r);
        \fill[gray!12] (C2r) circle (\r);
      \end{scope}
      \begin{scope}
        \clip (C1r) circle (\r);
        \draw[line width=0.8pt] (C2r) circle (\r);
      \end{scope}
      \begin{scope}
        \clip (C2r) circle (\r);
        \draw[line width=0.8pt] (C1r) circle (\r);
      \end{scope}
    \end{tikzpicture}
    \caption{}
    \label{fig:lens2}
  \end{subfigure}

  \caption{Two planar lenses}
  \label{fig:bothlenses}
\end{figure}

\subsection{How flat can counterexamples get in dimension $n\ge 3$}\label{sec:bad-stuff}

In the explicit example we provided, the radius of curvature in the symmetrized lens changed to approximately $1.03$. Thus one may wonder whether some approximate-convexity may hold, especially since a sketch of $\psi(s,t)=\sqrt{1+s^2+t^2}(1+s^2)$ defined before, might seem ``almost'' convex, at a glance. This however is not true, and for specific $(s,t)$ its convexity may fail very significantly, as the following lemma shows.

\begin{lem}\label{lem:psi-VERY-concave}
   Let $\psi:\RR^2\to\RR^+$ be defined by $\psi(s,t)= \sqrt{1+s^2+t^2}(1+s^2)$ as above. For any $R>0$ there exist $P_0=(t_0, s_0), P_1 = (t_1, s_1)$, satisfying
   \begin{equation}\label{eq:very-not-convex} \psi\left(\frac{P_0+P_1}{2}\right) \ge R \cdot \left(\frac{\psi(P_0) + \psi(P_1)}{2}\right).\end{equation}
\end{lem}

\begin{proof}
Fix $k>1$, and for any $t>0$ let $P_0  = (2t,0)$ and $P_1 = (0, 2t^{k})$, so that $\frac{1}{2}{(P_0 + P_1)} = (t,t^k)$. Next define 
\[ R(t) := \frac{\psi\left(\frac{P_0+P_1}{2}\right)}{\frac{\psi(P_0) + \psi(P_1)}{2}} 
=  \frac{2\sqrt{1+ t^2 + t^{2k}}(1+t^2)}{(1+4t^2)^{3/2} + \sqrt{1+4t^{2k}}}.
\]
For $k\in(1,3]$, $R(t)$ grows like $t^{k-1}$. For $k\ge 3$, $R(t)$ grows like $t^2$. In both cases we have
\[ 
\lim_{t\to \infty} R(t) = \lim_{t\to \infty}\frac{2\sqrt{1+ t^2 + t^{2k}}(1+t^2)}{(1+4t^2)^{3/2} + \sqrt{1+4t^{2k}}}
= \infty.
\]
Thus, picking $k>1$ and large enough $t$ we get points satisfying \eqref{eq:very-not-convex}. 
\end{proof}

To utilize this fact and construct more ``bad lenses'' we next 
show that we can find lenses with any prescribed outer normals at a given point. More precisely.
\begin{lem}\label{lem:lens-w-every-gradients}
Let $p,u,v\in \R^{n-1}$. Then, denoting by $P$ the projection from $\R^n$ to $\R^{n-1}$, there exists a lens $L\in\S_n$ of the form $\left\{(x,z)\,:\,x\in P(L),\,f(x)\le z \le g(x) \right\}$ such that $p \in P(L)$, $u=\nabla f(p)$ and $v=\nabla g(p)$.
\end{lem}

\begin{proof}
Without loss of generality we may assume $p = 0\in \RR^{n-1}$. As before, a unit ball centered at $(x_0,z_0)\in \RR^n$ is given by $\left\{(x,z)\,:\, -f_d(x)\le z\le f_u(x) \right\}$, for the functions $f_d,f_u,S:B(x_0, 1)\to\R$ that are given by $f_d(x) = -z_0 + S(x)$ and $f_u(x) = z_0 + S(x)$, where $S(x)=\sqrt{1-\|x-x_0\|^2}$ is half the secant length. The mapping $\nabla S$ from ${\rm int}(B(x_0, 1))\subset \RR^{n-1}$ to $\RR^{n-1}$ is a bijection, thus we may find $x_0$ such that for every $z_0$ we have $-\nabla  f_d (0) = u$.

Similarly, we consider a second ball, centered at $(y_0,0)\in \R^n$, which is given by $\left\{(x,z)\,:\, -g_d(x)\le z\le g_u(x) \right\}$, for $g_d(x) = g_u(x) = \sqrt{1-\|x-y_0\|^2}$. Again, we may find $y_0$ such that $\nabla  g_u (0) = v$.
Clearly $0\in {\rm int}( B(x_0, 1)\cap B(y_0, 1))\subset \R^{n-1}$.

We are left with picking $z_0$ such that the lens $L= B((x_0, z_0), 1)\cap B((y_0, 0), 1)$ has $(0, -f_d(0))$ and $(0, g_u(0))$ on its boundary. Clearly, the intersection of $L$ with the fiber $\R e_n$ is given by $\left\{(0,z)\,:\, \max(-f_d(0),-g_d(0))\le z \le \min(f_u(0),g_u(0))  \right\}$.

Since $0\in {\rm int}(B(y_0, 1))$, the segment $\RR e_n \cap B((y_0, 0), 1)=[-g_d(0),g_u(0)]e_n$ has positive length. Similarly the segment $\RR e_n \cap B((x_0, z_0), 1)=[-f_d(0),f_u(0)]e_n$ has positive length, and in fact the latter is also given by $[z_0-S(0),z_0+S(0)]e_n$. By choosing $z_0$ to be some positive value (e.g. $z_0=\max\left( S(0),g_u(0)\right)=\sqrt{1-\min\left( \|x_0\|^2,\|y_0\|^2\right)}~$), we can simultaneously have $-g_d(0)<-f_d(0)$ and $g_u(0)<f_u(0)$, guaranteeing that the lens $L$ has $(0, -f_d(0))$ and $(0, g_u(0))$ on its boundary, as required.
\end{proof}

With these two lemmas in hand, we may show that Steiner symmetrizations do not preserve the class of ball bodies in the following strong sense; any  assumption on large curvature of a ball body $K$, is not sufficient for deducing that its Steiner symmetral $S_u(K)$ is a ball body.

Fix some large $\kappa$, and consider the subclass of $\S_n$ consisting of bodies which are extremely curved, meaning all sectional curvatures are greater than $\kappa$. One may still find such a body $K$ and a direction $u$, such that the Steiner symmetral $S_u(K)$ is not a ball body.  

Since curvature is $1$-homogeneous, this amounts to finding, for any $\varepsilon>0$, a ball body $K$, whose symmetral $S_u(K)$ has a sectional curvature smaller than $\varepsilon$. We do this with a lens  (so the ``extremely curved'' body mentioned above, would be the intersection of two small balls of radius $\varepsilon\in (0,1)$). 

\begin{prop}
For any $\varepsilon>0$, there exists a lens $L\in \S_n$, such that the Steiner symmetral $S_{e_n}(L)$ has a point at which some sectional curvature is smaller than $\varepsilon$.
\end{prop}

\begin{proof}
By Lemma \ref{lem:psi-VERY-concave}, there exists $t>0$ such that the points $P_0=(2t,0), P_1=(0,2t^2)$ satisfy $\frac{\psi(P_0) + \psi(P_1)}{2} < \varepsilon\cdot \psi\left(\frac{P_0+P_1}{2}\right)$. Let $u=(2t,0,\dots,0), v=(0,\dots,0,2t^2)\in \R^{n-1}$. By Lemma \ref{lem:lens-w-every-gradients}, there exist concave functions $f, g$ defined on some convex $T\subset \R^{n-1}$ and a lens $L=\left\{(x,z)\,:\, x\in T,\,-f(x)\le z \le g(x)  \right\}$, such that $0\in T$, $-\nabla  f (0) = -u$, and $\nabla  g (0) = v$.

The Steiner symmetral $S_{e_n}(L)$ is given by $\left\{(x,z)\,:\, x\in T,\,-h(x)\le z \le h(x)  \right\}$ for $h = \frac{f + g}{2}$. Similarly to \eqref{eq:sec-curv}, the sectional curvature of $S_{e_n}(L)$ in the points $(x,\pm h(x))$, in direction $e_1$, is given  by 
\[ \kappa_h (x,e_1) = \frac{|(\nabla^2 h (x))_{1,1}|}{\psi\left(\iprod{e_1}{\nabla h(x)},\sqrt{\|\nabla h(x)\|^2 - \iprod{e_1}{\nabla h(x)}^2}\right)}.
\]
At the point $x=0$, we get
\[ \kappa_h (0,e_1) = \frac{|(\nabla^2 h (0))_{1,1}|}{\psi\left(t,t^2\right)},
\]
since $\nabla h(0)=\frac{u+v}2=(t,0,\dots,0,t^2)$ (recall, $\nabla  f (0) = u$, and $\nabla  g (0) = v$). As in the proof of Theorem \ref{thm:counter1}, the curvatures of the spheres $f,g$ equal $1$, so we have 
\[|(\nabla^2 f (0))_{1,1}|=\psi\left(\iprod{e_1}{\nabla f(0)},\sqrt{\|\nabla f(0)\|^2 - \iprod{e_1}{\nabla f(0)}^2}\right)
=\psi(P_0),\]
\[|(\nabla^2 g (0))_{1,1}|=\psi\left(\iprod{e_1}{\nabla g(0)},\sqrt{\|\nabla g(0)\|^2 - \iprod{e_1}{\nabla g(0)}^2}\right)
=\psi(P_1).\]
Combining the above, we get
\[
\kappa_h (0,e_1) = \frac{|(\nabla^2 h (0))_{1,1}|}{\psi\left(t,t^2\right)}
= \frac{\frac{|(\nabla^2 f (0))_{1,1}|+|(\nabla^2 g (0))_{1,1}|}2}{\psi\left(\frac{P_0+P_1}{2}\right)}
= \frac{\frac{\psi(P_0)+\psi(P_1)}2}{\psi\left(\frac{P_0+P_1}{2}\right)}
<\varepsilon.
\]
\end{proof}

\section*{Appendix}
This section is devoted to the proof of the technical statement below, which can be of use when considering ball body inequalities, and which we have used in our second proof of Theorem \ref{thm:InR2Steiner-preserves-class}. 

\begin{prop}[Denseness of smooth bodies in $\S_n$]\label{thm:dense-are-the-smooth}
Let $n\in \mathbb N$. For any $K\in \S_n$ there exists a sequence $(K_m)_{m\in \mathbb N}$ with $K_m\in \S_n$ which are $C^\infty$ smooth convex bodies with $h_K\in C^\infty$, such that  $d_H(K_m,K)\to_{m\to \infty} 0$. 
\end{prop}

 To prove this proposition we employ standard approximation techniques which are explained\footnote{We would like to thank Daniel Hug for discussing approximations and for pointing us to the most relevant theorem in \cite{schneider2013convex}} clearly in \cite[Section 3.4]{schneider2013convex}.

\begin{proof}[Proof of Theorem \ref{thm:dense-are-the-smooth}]
    Given $K\in \S_n$ we may assume without loss of generality that $K\subseteq B_2^n$. 
    We employ the approximation procedure described in \cite[Theorem 3.4.1]{schneider2013convex}, where  $\eps_m>0$ is some sequence with $\eps_m \to 0$. To this end we fix for every $m$ some $\varphi_m:[0,\infty)\to [0,\infty)$ which is $C^\infty$ smooth, has $\int \varphi_m\left(\|z\|\right)dz = 1$,  and is supported on $[\eps_m/2, \eps_m)$, and define the mapping
    \[ T_mf (x) = \int_{\RR^n} f\left(x + \|x\|z\right)\varphi_m\left(\|z\|\right)dz.\]
Theorem 3.4.1 in \cite{schneider2013convex} implies that $T_m(h_{K})$ is the 
    support function of a convex body, which we call $K_m'$, and moreover $h_{K_m'}$
     is $C^\infty$ on $\RR^n\setminus \{0\}$. 
     Moreover (upon identifying the map $T_m$ on support functions and on convex bodies), $d_H(K, T_mK) \le \eps_m$ for all $K\in \S_n$ since $\outrad(K)\le 1$ and using property (c) of \cite[Theorem 3.4.1]{schneider2013convex}. We see also that $T_m(K+L) = T_mK + T_mL$ by definition (this is (a) in \cite[Theorem 3.4.1]{schneider2013convex}), and  that $T_m(B_2^n)$ is a Euclidean ball since it is invariant under rigid motions by property (b) of the same theorem. Denoting $T_m(B_2^n) = \alpha_m B_2^n$, we see that $\alpha_m \in [1-\eps_m, 1+\eps_m]$ since $d_H(\alpha_mB_2^n , B_2^n)\le \eps_m$.   
In particular, if $K+L= B_2^n$ we have $K_m'+T_m(L)= \alpha_m B_2^n$ and thus $\frac{1}{\alpha_m} K_m' + \frac{1}{\alpha_m}T_m(L) = B_2^n$. We let $K_m'' = \frac{1}{\alpha_m} K_m'$, so that 
$d_H(K_m'', K_m') = |1-\frac{1}{\alpha_m}|\sup_{u\in S^{n-1}}|h_{K_m'}(u)|\le |1-\frac{1}{\alpha_m}|\le 2\eps_m$ if we assume $\eps_m<1/2$, which we may. 
We see that $K_m''$ is a summand of $B_2^n$, and its support function is $C^\infty$, since this was the case for $K_m'$. 
The last step in our construction is to ensure that the body has no singular points. This would already imply that the body is $C^\infty$; For the discussion connecting the smoothness of the support function with the smoothness of the body, see \cite[Section 2.5]{schneider2013convex}  where the $C^2_+$ case is considered, but the proof works for any degree of smoothness. See also the discussion after Theorem 3.4.1 in the same book. To this end
we let 
\[
K_m  = (1-\eps_m)K_m'' + \eps_m B_2^n, 
\]
so that 
\[ d_H(K_m'',K_m)= \sup_{u\in S^{n-1}} |h_{K_m}(u) - h_{K_m''}(u)|
= \eps_m d_H(K_m'', B_2^n) \le \eps_m.
\] 
Clearly $K_m$ is a summand of $B_2^n$,  it is $C^\infty_+$, and its support function is also $C^\infty$. We get that  
\[ 
d_H(K_m, K) \le d_H(K_m, K_m'') + d_H(K_m'', K_m')+d_H(K_m',  K)\le 4\eps_m  
\]
and the proof is complete. 
\end{proof}

\bibliographystyle{plain}

{\small
\noindent S. Artstein-Avidan, 
\vskip 2pt
\noindent School of Mathematical Sciences, Tel Aviv University, Ramat
Aviv, Tel Aviv, 69978, Israel.\vskip 2pt
\noindent Email: shiri@tauex.tau.ac.il.
\vskip 2pt
\noindent D.I. Florentin, 
\vskip 2pt
\noindent Department of Mathematics, Bar-Ilan University, Ramat Gan,  52900, Israel.   \vskip 2pt
\noindent Email: dan.florentin@biu.ac.il.
}

\end{document}